\documentclass[11pt,A4]{article}
\usepackage[latin5]{inputenc}
\usepackage{amsmath,amssymb}
\usepackage{amsthm}
\usepackage{indentfirst}

\newtheorem{definition}{Definition}
\newtheorem{theorem}{Theorem}

\newtheorem{corollary}{Corollary}

\newtheorem{proposition}{Proposition}

\numberwithin{definition}{section} \numberwithin{theorem}{section}
\numberwithin{lemma}{section}\numberwithin{corollary}{section}
\numberwithin{equation}{section} \numberwithin{example}{section}
\numberwithin{proposition}{section} \numberwithin{remark}{section}
\begin{document}

\begin{center}

{\Large Approximation of the Image of the Hilbert-Schmidt \\ Integral Operator}

\vspace{5mm}

Nesir Huseyin

\vspace{5mm}

{\small Sivas Cumhuriyet University, Department Mathematics and Science Education \\  58140 Sivas, Turkey

\vspace{3mm}

E-mail: nhuseyin@cumhuriyet.edu.tr}

\end{center}

\vspace{5mm}

\textbf{Abstract.}
In this paper an approximation of the image of the closed ball of the space $L_p$ $(p>1)$ centered at the origin with radius $r$ under Hilbert-Schmidt integral operator $F(\cdot):L_p\rightarrow L_q$ $\displaystyle \left(\frac{1}{p}+\frac{1}{q}=1\right)$  is presented. An error estimation for given approximation is obtained.

\vspace{3mm}

\vspace{3mm}

\section{Introduction}

Integral operators arise in various problems of theory and applications and is one of the important tools to investigate different type problems of mathematics. For example, integral operators are used in  Fredholm, Volterra, Urysohn-Hammerstein and etc. type integral equations and play crucial role for definition of solution concepts for different type initial and boundary value problems of differential equations (see, e.g.  \cite{cor}, \cite{kras}). It is necessary to underline that the theory of linear integral equations are considered one of the origins of the contemporary functional analysis (\cite{goh}, \cite{hil}, \cite{rie}). In particular, the integral operators are used for description the behaviour of some input-output systems (see, \cite{hus}, \cite{pol} and references therein).

In present paper an approximation of the image of the closed ball of the space $L_p$ $(p>1)$ centered at the origin under Hilbert-Schmidt integral operator is considered. Presented approximation method allows for every $\varepsilon >0$ to construct a finite $\varepsilon$-net on the image of the closed ball which consists of the images of a finite number of piecewise constant functions. Approximation of the image of given closed ball can be used in infinite dimensional optimization problems for predetermining the desirable inputs for input-output system described by Hilbert-Schmidt integral operator. Note that integrally constrained inputs are usually applied when the input resources of the system are exhausted by consumption such as energy, fuel, finance, etc. (see, \cite{hus}, \cite{hus1}, \cite{kra} and references therein). An evaluation of error estimation for Hausdorff distance between the image and its approximation, which consists of a finite number of functions, is given.

The paper is organized as follows. In Section 2 the conditions and auxiliary propositions which are used in following arguments, are formulated. In Section 3 the image of the integral operator is approximated by the set, consisting of a finite number of functions. An evaluation of error estimation depending on the approximation parameters is given (Theorem \ref{teo1}).

\section{Preliminaries}

Consider the Hilbert-Schmidt integral operator
\begin{eqnarray}\label{fo}
F(x(\cdot))| (\xi)= \int_{\Omega} K(\xi,s)x(s) ds \ \ \mbox{for almost all} \ \ \xi \in \Omega,
\end{eqnarray}
where $x(s)\in \mathbb{R}^n$, $K(\xi,s)$ is $m\times n$ dimensional matrix function,  $(\xi,s) \in \Omega\times \Omega,$ $\Omega \subset \mathbb{R}^k$ is a compact set.

For given $p>1$ and $r>0$ we denote
\begin{eqnarray*}\label{con}
B_{p}(r)=\left\{ x(\cdot)\in L_p \left( \Omega ;\mathbb{R}^n \right): \left\|x(\cdot)\right\|_p \leq r\right\},
\end{eqnarray*} where $L_p\left(\Omega;\mathbb{R}^n \right)$ is the space of Lebesgue measurable functions $x(\cdot):\Omega\rightarrow \mathbb{R}^n$ such that $\left\|x(\cdot)\right\|_p <+\infty$, $\displaystyle \left\|x(\cdot)\right\|_p =\left( \int_{\Omega} \left\|x(s)\right\|^p ds \right)^{\frac{1}{p}}$, $\left\| \cdot \right\|$ denotes the Euclidean norm.

It is assumed that the matrix function $K(\cdot, \cdot):\Omega \times \Omega \rightarrow \mathbb{R}^{m\times n}$ is Lebesgue measurable and
\begin{eqnarray*}\label{k*}
\displaystyle \int_{\Omega}\int_{\Omega} \left\|K(\xi,s)\right\|^q d\xi ds < +\infty,
\end{eqnarray*} where $\displaystyle \frac{1}{q} +\frac{1}{p}=1$. Denote
\begin{eqnarray}\label{eq1}
\mathcal{F}_p(r)= \left\{ F(x(\cdot))|(\cdot): x(\cdot)\in B_p(r)\right\}.
\end{eqnarray}

It is obvious that the set $\mathcal{F}_p(r)$ is the image of the set $B_{p}(r)$ under Hilbert-Schmidt integral operator (\ref{fo}). Since operator $F(\cdot)$ is linear and compact one, then we have that the set $\mathcal{F}_p(r)$ is a convex and compact subset of the space $L_q\left(\Omega;\mathbb{R}^m \right)$.

Since the set of continuous functions $\Phi(\cdot,\cdot):\Omega \times \Omega \rightarrow \mathbb{R}^{m\times n}$ is dense in the space $L_q\left(\Omega \times \Omega; \mathbb{R}^{m\times n}\right)$ (see, e.g. \cite{kan}, p.318), then for every $\lambda >0$  there exists a continuous function $K_{\lambda}(\cdot, \cdot):\Omega \times \Omega \rightarrow \mathbb{R}^{m\times n}$ such that
\begin{eqnarray}\label{eq1a}
\left(\int_{\Omega}\int_{\Omega} \left\|K(\xi,s)-K_{\lambda}(\xi,s)\right\|^q d\xi ds\right)^{\frac{1}{q}} \leq \frac{\lambda}{2r} \, .
\end{eqnarray} Denote
\begin{eqnarray}\label{eq2a}
M(\lambda)=\max \left\{\left\|K_{\lambda}(\xi,s)\right\|: (\xi,s)\in \Omega \times \Omega\right\},
\end{eqnarray}
\begin{eqnarray}\label{eq3a}
\omega_{\lambda}(\Delta)&=&\max \big\{\left\|K_{\lambda}(\xi,s_2)-K_{\lambda}(\xi,s_1)\right\|: (\xi,s_2)\in \Omega \times \Omega, \nonumber \\ & \ \ \ & (\xi,s_1)\in \Omega \times \Omega, \ \left\|s_2-s_1\right\|\leq \Delta\big\},
\end{eqnarray} where $\Delta>0$ is a given number. Compactness of the set $\Omega \subset \mathbb{R}^k$ and continuity of the function $K_{\lambda}(\cdot, \cdot):\Omega \times \Omega \rightarrow \mathbb{R}^{m\times n}$ imply that for each fixed $\lambda >0$ we have $\omega_{\lambda} (\Delta) \rightarrow 0$ as $\Delta \rightarrow 0^+$  and $\omega_{\lambda} (\Delta_1) \leq \omega_{\lambda} (\Delta_2)$ if $\Delta_1 <\Delta_2.$

\vspace{2mm}

Let us define a finite $\Delta$-partition of the given compact set $\Omega \subset \mathbb{R}^k$ which will be used in following arguments.
\begin{definition}\label{def2.1} Let $\Delta >0$ and $E\subset \mathbb{R}^k.$ A finite family of sets $\Lambda=\left\{E_1,E_2, \ldots, E_l\right\}$ is called a finite $\Delta$-partition of the set $E$, if
\item{1.} $E_i \subset E$ and $E_i$ is Lebesgue measurable for every $i=1,2,\ldots ,l$;
\item{2.} $E_i\bigcap E_j =\emptyset$ for every $i\neq j$, where $i=1,2,\ldots, l$ and $j=1,2,\ldots, l$;
\item{3.} $E =\bigcup_{i=1}^{l} E_i$;
\item{4.} $diam \, (E_i) \leq \Delta$ for every $i=1,2,\ldots, l,$ where $diam \, (E_i)=\sup\big\{\left\|x-y\right\|:$ $x\in E_i, \ y\in E_i \big\}.$
\end{definition}
\begin{proposition} \label{prop2.1} Let $\Omega \subset \mathbb{R}^k$ be a compact set. Then for every $\Delta >0$ it has a finite $\Delta$-partition $\Lambda=\left\{\Omega_1, \Omega_2, \ldots, \Omega_N\right\}.$
\end{proposition}

\section{Approximation}

Let $\gamma>0$,   $\Lambda= \left\{\Omega_1,\Omega_2,\ldots, \Omega_N \right\}$ be a finite $\Delta$-partition of the compact set $\Omega\subset \mathbb{R}^k,$ $\Lambda_* =\left\{0=z_0, z_1, \ldots , z_a=\gamma \right\}$ be a uniform partition of the closed interval $[0,\gamma],$ $\delta=z_{j+1}-z_j,$ $j=0,1,\ldots, a-1,$ be a diameter of the partition $\Lambda_*,$ $\sigma >0$ be a given number, $E=\left\{x\in \mathbb{R}^n: \left\|x\right\| =1\right\}$ and $E_{\sigma}=\left\{e_1,e_2, \ldots, e_c\right\}$  be a finite $\sigma$-net on $E$. Denote
\begin{eqnarray}\label{fin} \displaystyle && B_{p}^{\gamma,\Delta,\delta,\sigma}(r)=\Big\{ x(\cdot):\Omega \rightarrow \mathbb{R}^n :  x(\xi)=z_{j_i}e_{l_i} \ \mbox{for every} \
\xi \in \Omega_i, \  \mbox{where} \nonumber \\ && \qquad \displaystyle z_{j_i} \in \Lambda_*, \ e_{l_i} \in E_{\sigma}, \
i=1,2,\ldots,  N, \  \sum_{i=1}^{N}\mu(\Omega_i) z_{j_i}^p \leq r^p  \Big\},
\end{eqnarray}
\begin{eqnarray}\label{eq801}
\mathcal{F}_{p}^{\gamma, \Delta, \delta, \sigma}(r)= \left\{ F(x(\cdot))|(\cdot): x(\cdot)\in B_{p}^{\gamma, \Delta, \delta, \sigma}(r)\right\},
\end{eqnarray}
where $\mu(\cdot)$ means the Lebesgue measure of a set. It is obvious that the set $\mathcal{F}_{p}^{\gamma, \Delta, \delta, \sigma}(r)$ consists of a finite number functions. We set
\begin{eqnarray}\label{c*}
c_*= 2r^p \left[\mu (\Omega)\right]^{\frac{1}{q}} ,
\end{eqnarray}
\begin{eqnarray}\label{psi}
\psi_{\lambda}(\Delta)=  2r \left[\mu (\Omega)\right]^{\frac{2}{q}} \omega_{\lambda}(\Delta),
\end{eqnarray}
\begin{eqnarray}\label{phi}
\varphi_{\lambda}(\delta)= M(\lambda)\left[\mu (\Omega)\right]^{1+\frac{1}{q}}\delta,
\end{eqnarray}
\begin{eqnarray}\label{alfa}
\alpha_{\lambda}(\gamma, \sigma)=M(\lambda)\left[\mu(\Omega)\right]^{1+\frac{1}{q}}  \gamma \sigma
\end{eqnarray} where $M(\lambda)$ is defined by (\ref{eq2a}).

The Hausdorff distance between the sets $U \subset L_q\left(\Omega;\mathbb{R}^m \right)$ and $V \subset L_q\left(\Omega;\mathbb{R}^m \right)$ is denoted by $h_{q}(U,V)$.
\begin{theorem} \label{teo1} For every $\lambda >0,$ $\gamma>0,$ $\Delta$-partition of the compact set $\Omega\subset \mathbb{R}^k,$ $\delta$-partition of the closed interval $[0,\gamma]$ and $\sigma >0$ the inequality
\begin{eqnarray*}
h_{q}\left(\mathcal{F}_{p}(r),\mathcal{F}_{p}^{\gamma, \Delta, \delta, \sigma}(r)\right) \leq \lambda+ \frac{c_* M(\lambda)}{\gamma^{p-1}}+ \psi_{\lambda} (\Delta) +\varphi_{\lambda}(\delta) +\alpha_{\lambda}(\gamma, \sigma)
\end{eqnarray*}
is satisfied, where the sets $\mathcal{F}_{p}(r)$ and $\mathcal{F}_{p}^{\gamma, \Delta, \delta, \sigma}(r)$ are defined by (\ref{eq1}) and (\ref{eq801}) respectively.
\end{theorem}
\begin{proof} The proof of the theorem will be completed in 7 steps.

\emph{Step 1.}
 Denote
\begin{eqnarray}\label{fo*}
F_{\lambda}(x(\cdot))| (\xi)= \int_{\Omega} K_{\lambda}(\xi,s)x(s) ds  \ \ \mbox{for every} \ \ \xi \in \Omega,
\end{eqnarray} and
\begin{eqnarray*}
\mathcal{F}_{p}^{\lambda}(r)= \left\{ F_{\lambda}(x(\cdot))|(\cdot): x(\cdot)\in B_p(r)\right\}.
\end{eqnarray*} where $K_{\lambda}(\cdot,\cdot)$ is defined in (\ref{eq1a}).
The set $\mathcal{F}_{p}^{\lambda}(r)$ is the image of the set $B_{p}(r)$ under Hilbert-Schmidt integral operator (\ref{fo*}) and compactness of the operator $F_{\lambda}(\cdot)$ implies that  the set $\mathcal{F}_{p}^{\lambda}(r)$ is a compact subset of the space $C\left(\Omega;\mathbb{R}^m \right)$, where $C\left(\Omega;\mathbb{R}^m \right)$ is the space of continuous functions $x(\cdot):\Omega \rightarrow \mathbb{R}^m$ with norm
$\left\| x(\cdot)\right\|_C=\max \left\{ \left\|x(\xi)\right\|: \xi \in \Omega\right\}.$

Applying (\ref{eq1a}) and H\"{o}lder's inequality it is not difficult to show that
\begin{eqnarray}\label{eq1t}
h_{q}( \mathcal{F}_p(r),\mathcal{F}_{p}^{\lambda}(r)) \leq \frac{\lambda}{2} \, .
\end{eqnarray}

\vspace{3mm}

\emph{Step 2.} Denote
 \begin{eqnarray*}
\displaystyle B_{p}^{\gamma}(r)=\left\{ x(\cdot) \in B_{p}(r): \left\|x(\xi) \right\| \leq \gamma \ \mbox{for every} \ \xi \in \Omega\right\}
\end{eqnarray*} and let
\begin{eqnarray*}
\mathcal{F}_{p}^{\lambda, \gamma}(r)= \left\{ F_{\lambda}(x(\cdot))|(\cdot): x(\cdot)\in B_{p}^{\gamma}(r)\right\}.
\end{eqnarray*}

Let $y_*(\cdot)\in \mathcal{F}_{p}^{\lambda}(r)$ be an arbitrary chosen function which is the image of $x_*(\cdot)\in B_{p}(r)$ under operator (\ref{fo*}). Define the function $x_0(\cdot):\Omega\rightarrow \mathbb{R}^n$ setting
\begin{eqnarray*} \label{eq806}
x_{0}(s)=\left\{
 \begin{array}{llll}
    x_*(s) \ ,    & \mbox{if} \ \left\Vert x_*(s) \right\Vert \leq \gamma, \\
   \displaystyle \gamma \frac{x_*(s)}{\left\Vert x_*(s) \right\Vert} \ , & \mbox{if} \ \left\Vert x_*(s) \right\Vert > \gamma
 \end{array}
 \right.
\end{eqnarray*}
where $s \in \Omega$. It  is not difficult to verify that  $x_{0}(\cdot) \in B_{p}^{\gamma}(r).$ Let $y_0(\cdot)\in \mathcal{F}_{p}^{\lambda,\gamma}(r)$ be the image of $x_{0}(\cdot) \in B_{p}^{\gamma}(r)$ under operator (\ref{fo*}). Denote $W=\left\{s \in \Omega: \left\Vert x_*(s) \right\Vert > \gamma \right\}.$
From inclusion $x_*(\cdot)\in B_{p}(r)$ and Tchebyshev's inequality (see, \cite{whe}, p.82) it follows that
\begin{eqnarray}\label{eq807*}
\mu(W)  \leq \frac{r^p}{\gamma^p}.
\end{eqnarray}

Thus, from (\ref{c*}), (\ref{eq807*}) and H\"{o}lder's inequality we obtain that
 \begin{eqnarray*}
\displaystyle \left\|y_*(\xi) -y_0(\xi) \right\| &\leq &  \int_{W}  \left\|K_{\lambda} (\xi,s)\right\| \left\|x_*(s)-x_0(s)\right\| ds \nonumber \\ & \leq & 2rM(\lambda)[\mu(W)]^{\frac{1}{q}} \leq \frac{2r^p M(\lambda)}{\gamma^{p-1}}
\end{eqnarray*} for every $\xi \in \Omega$ and consequently
\begin{eqnarray*}  \displaystyle
\left\| y_*(\cdot) -y_0(\cdot)\right\|_{q} \leq \frac{2r^p M(\lambda)}{\gamma^{p-1}}\left[\mu (\Omega)\right]^{\frac{1}{q}} =  \frac{c_* M(\lambda)}{\gamma^{p-1}},
\end{eqnarray*}
where $M(\lambda)$ is defined by (\ref{eq2a})
Since $y_*(\cdot)\in \mathcal{F}_{p}^{\lambda}(r)$ is arbitrarily chosen, we obtain from the last inequality that
\begin{eqnarray}\label{eq809}
\displaystyle \mathcal{F}_{p}^{\lambda}(r) \subset \mathcal{F}_{p}^{\lambda, \gamma}(r) +\frac{c_* M(\lambda)}{\gamma^{p-1}}B_q(1).
\end{eqnarray}

The inclusion $\mathcal{F}_{p}^{\lambda, \gamma}(r) \subset \mathcal{F}_{p}^{\lambda}(r)$ and (\ref{eq809}) yield that
\begin{eqnarray}\label{eq810}
h_{q} \left(\mathcal{F}_{p}^{\lambda}(r),  \mathcal{F}_{p}^{\lambda, \gamma}(r)\right) \leq \frac{c_* M(\lambda)}{\gamma^{p-1}}.
\end{eqnarray}

\vspace{3mm}

\emph{Step 3.} For given $\Delta >0$ and finite $\Delta$-partition $\Lambda= \left\{\Omega_1,\Omega_2,\ldots, \Omega_N \right\}$ of the compact set $\Omega\subset \mathbb{R}^k$ we denote
 \begin{eqnarray}\label{eq810*}
\displaystyle B_{p}^{\gamma, \Delta}(r)&=&\big\{ x(\cdot) \in B_{p}^{\gamma}(r): x(\xi) =x_i \nonumber \\ && \mbox{for every} \ \xi \in \Omega_i, \ i=1,2,\ldots ,N\big\}
\end{eqnarray} and
\begin{eqnarray*}
\mathcal{F}_{p}^{\lambda, \gamma, \Delta}(r)= \left\{ F_{\lambda}(x(\cdot))|(\cdot): x(\cdot)\in B_{p}^{\gamma,\Delta}(r)\right\}.
\end{eqnarray*}

Choose an arbitrary $\tilde{y}(\cdot)\in \mathcal{F}_{p}^{\lambda,\gamma}(r)$ which is the image of $\tilde{x}(\cdot)\in B_{p}^{\gamma}(r)$ under operator (\ref{fo*}). Define the function $\tilde{x}_*(\cdot):\Omega\rightarrow \mathbb{R}^n$ setting
\begin{eqnarray}\label{eq813}
\displaystyle \tilde{x}_*(\xi) =\frac{1}{\mu(\Omega_i)}\int_{\Omega_i} \tilde{x}(s)ds, \ \xi \in \Omega_i,  \ i=1,2,\ldots , N.
\end{eqnarray}

The inclusion $\tilde{x}(\cdot)\in B_{p}^{\gamma}(r)$ gives us that $\left\|\tilde{x}(s)\right\| \leq \gamma$ for every $s \in \Omega.$ Then from (\ref{eq813}) it follows that
\begin{eqnarray}\label{eq814}
\left\|\tilde{x}_*(\xi)\right\| \leq \gamma \ \mbox{for every} \ \xi \in \Omega.
\end{eqnarray}

 (\ref{eq813}) and H\"{o}lder's inequality yield that
\begin{eqnarray*}
\displaystyle \left\|\tilde{x}_*(\xi)\right\| \leq \frac{1}{\left[\mu(\Omega_i)\right]^{\frac{1}{p}}}\left(\int_{\Omega_i} \left\|\tilde{x}(s)\right\|^p ds\right)^{\frac{1}{p}}, \ \xi \in \Omega_i,  \ i=1,2,\ldots , N
\end{eqnarray*}
and consequently
\begin{eqnarray*}
\displaystyle \int_{\Omega_i} \left\|\tilde{x}_*(s)\right\|^pds \leq \int_{\Omega_i} \left\|\tilde{x}(s)\right\|^pds
\end{eqnarray*}
for every $ i=1,2,\ldots , N.$ Since $\left\|\tilde{x}(\cdot)\right\|_p \leq r$, then it follows from last inequality that
\begin{eqnarray}\label{eq815}
\displaystyle \int_{\Omega} \left\|\tilde{x}_*(s)\right\|^pds &=&\sum_{i=1}^{N} \int_{\Omega_i} \left\|\tilde{x}_*(s)\right\|^pds \leq  \sum_{i=1}^{N} \int_{\Omega_i} \left\|\tilde{x}(s)\right\|^pds \nonumber \\ &=& \int_{\Omega} \left\|\tilde{x}(s)\right\|^pds \leq r^p.
\end{eqnarray}

From (\ref{eq810*}), (\ref{eq813}), (\ref{eq814}) and (\ref{eq815}) we obtain  that  $\tilde{x}_*(\cdot)\in B_{p}^{\gamma,\Delta}(r).$ Let
$\tilde{y}_*(\cdot)\in \mathcal{F}_{p}^{\lambda, \gamma,\Delta}(r)$ be the image of $\tilde{x}_*(\cdot)$ under operator (\ref{fo*}). We have
\begin{eqnarray}\label{eq817}
\displaystyle \left\|\tilde{y}_*(\xi)-\tilde{y}(\xi)\right\| & = &
\left\| \int_{\Omega} K_{\lambda}(\xi,s)\left[\tilde{x}_*(s)-\tilde{x}(s)\right]ds \right\| \nonumber \\
&= & \left\| \sum_{i=1}^{N}\int_{\Omega_i} K_{\lambda}(\xi,s)\left[\tilde{x}_*(s)-\tilde{x}(s)\right]ds \right\|
\end{eqnarray} for every $\xi \in \Omega.$
(\ref{eq813}) yields that
\begin{eqnarray}\label{eq816*}
\displaystyle \int_{\Omega_i} \tilde{x}_*(s) ds = \int_{\Omega_i} \tilde{x}(s)ds
\end{eqnarray} for every $i=1,2,\ldots, N.$ Let $\xi \in \Omega$ and $i=1,2,\ldots,N$ be fixed. Now let us choose an arbitrary $s_i \in \Omega_i.$  From (\ref{eq816*}) it follows that
\begin{eqnarray}\label{eq817*}
 \displaystyle && \left\| \int_{\Omega_i} K_{\lambda}(\xi,s)\left[\tilde{x}_*(s))-\tilde{x}(s)\right]ds \right\| \nonumber \\ && \displaystyle \qquad =\left\|\int_{\Omega_i}\left[K_{\lambda}(\xi,s)- K_{\lambda}(\xi,s_i)\right]\left[\tilde{x}_*(s)-\tilde{x}(s)\right]ds \right\| \nonumber \\ && \displaystyle \qquad \leq \int_{\Omega_i}\left\|K_{\lambda}(\xi,s)- K_{\lambda}(\xi,s_i)\right\|\left\|\tilde{x}_*(s)-\tilde{x}(s)\right\|ds.
\end{eqnarray}

Since $\Lambda=\left\{\Omega_1, \Omega_2, \ldots, \Omega_N\right\}$ is a $\Delta$-partition of $\Omega,$  $s_i\in \Omega_i$, then from Definition \ref{def2.1}  we obtain that $\left\|s-s_i\right\|\leq \Delta$ for every $s\in \Omega_i.$ Finally, by virtue of (\ref{eq3a}) we have that
\begin{eqnarray}\label{eq819}
\left\|K_{\lambda}(\xi,s)-K_{\lambda}(\xi,s_i)\right\| \leq  \omega_{\lambda} (\Delta)
\end{eqnarray}
for every $s\in \Omega_i.$ Thus, from (\ref{eq817*}) and (\ref{eq819}) it follows that
\begin{eqnarray}\label{eq820}
\displaystyle \left\| \int_{\Omega_i} K_{\lambda}(\xi,s)\left[\tilde{x}_*(s)-\tilde{x}(s)\right]ds \right\|   \leq \omega_{\lambda}(\Delta)
\int_{\Omega_i}\left\|\tilde{x}_*(s)-\tilde{x}(s)\right\|ds.
\end{eqnarray}  Since $\tilde{x}_*(\cdot)\in B_{p}(r)$ and $\tilde{x}(\cdot)\in B_{p}(r),$ then (\ref{eq820}) and H\"{o}lder's inequality yield that
\begin{eqnarray}\label{eq821}
&& \displaystyle \left\| \sum_{i=1}^{N} \int_{\Omega_i} K_{\lambda}(\xi,s)\left[\tilde{x}_*(s)-\tilde{x}(s)\right]ds \right\| \leq \omega_{\lambda} \left(\Delta \right) \sum_{i=1}^{N}
\int_{\Omega_i}\left\|\tilde{x}_*(s)-\tilde{x}(s)\right\|ds \nonumber \\ && \qquad = \omega_{\lambda} \left(\Delta \right)
\int_{\Omega}\left\|\tilde{x}_*(s)-\tilde{x}(s)\right\|ds  \leq 2 \omega_{\lambda}\left(\Delta \right) \left[\mu(\Omega)\right]^{\frac{1}{q}}r.
\end{eqnarray} (\ref{psi}), (\ref{eq817}) and (\ref{eq821}) imply that
\begin{eqnarray*}
\displaystyle \left\|\tilde{y}_*(\xi)-\tilde{y}(\xi)\right\| \leq 2\omega_{\lambda}\left(\Delta \right) \left[\mu(\Omega)\right]^{\frac{1}{q}}r
\end{eqnarray*}
for every $\xi \in \Omega$ and consequently
\begin{eqnarray*}\label{eq822}  \displaystyle
\left\| \tilde{y}_*(\cdot) -\tilde{y}(\cdot)\right\|_{q} \leq 2r \omega_{\lambda}(\Delta)\left[\mu (\Omega)\right]^{\frac{2}{q}} =  \psi_{\lambda} (\Delta) .
\end{eqnarray*}

Since $\tilde{y}(\cdot)\in \mathcal{F}_{p}^{\lambda, \gamma}(r)$ is arbitrarily chosen, the last inequality yields
\begin{eqnarray}\label{eq823}
\displaystyle \mathcal{F}_{p}^{\lambda,\gamma}(r) \subset \mathcal{F}_{p}^{\lambda, \gamma,\Delta}(r) +\psi_{\lambda} (\Delta) B_q(1).
\end{eqnarray}

From inclusion $\mathcal{F}_{p}^{\lambda, \gamma, \Delta}(r) \subset \mathcal{F}_{p}^{\lambda,\gamma}(r)$ and (\ref{eq823}) we obtain
\begin{eqnarray}\label{eq824}
h_{q} \left(\mathcal{F}_{p}^{\lambda,\gamma}(r),  \mathcal{F}_{p}^{\lambda, \gamma,\Delta}(r)\right) \leq  \psi_{\lambda} (\Delta).
\end{eqnarray}

\vspace{3mm}

\emph{Step 4.} For given $\Delta >0,$ $\delta>0,$ finite $\Delta$-partition $\Lambda= \left\{\Omega_1,\Omega_2,\ldots, \Omega_N \right\}$ of the compact set $\Omega\subset \mathbb{R}^k$ and uniform $\delta$-partition $\Lambda_* =\left\{0=z_0, z_1, \ldots , z_a=\gamma \right\}$ of the closed interval $[0,\gamma]$  we set
\begin{eqnarray*}
\displaystyle B_{p}^{\gamma, \Delta, \delta}(r) &=&\big\{ x(\cdot) \in B_{p}^{\gamma,\Delta}: x(\xi)=x_i \ \mbox{for every} \ \xi \in \Omega_i, \\ && \qquad \left\|x_i\right\| \in \Lambda_* \ \mbox{for every} \ i=1,2,\ldots , N \big\},
\end{eqnarray*}
\begin{eqnarray*}
\mathcal{F}_{p}^{\lambda, \gamma, \Delta, \delta}(r)= \left\{ F_{\lambda}(x(\cdot))|(\cdot): x(\cdot)\in B_{p}^{\gamma,\Delta,\delta}(r)\right\}.
\end{eqnarray*}

Let $\hat{y}_0(\cdot)\in \mathcal{F}_{p}^{\lambda,\gamma, \Delta}(r)$ be an arbitrary chosen function which is the image of  $\hat{x}_0(\cdot)\in B_{p}^{\gamma, \Delta}(r)$ under operator (\ref{fo*}). From inclusion $\hat{x}_0(\cdot)\in B_{p}^{\gamma, \Delta}(r)$ it follows that
\begin{eqnarray}\label{eq903}
\displaystyle \tilde{x}_0(\xi) =x_i, \ \xi \in \Omega_i,  \ i=1,2,\ldots , N,
\end{eqnarray} where
\begin{eqnarray}\label{eq904}
\displaystyle  \sum_{i=1}^{N} \mu(\Omega_i)\left\|x_i\right\|^p \leq r^p, \ \left\|x_i\right\| \leq \gamma \ \mbox{for every}  \ i=1,2,\ldots , N
\end{eqnarray}

Since $\left\|x_i\right\| \in [0,\gamma]$ for every $i=1,2,\ldots , N$, then $\left\|x_i\right\| <\gamma$ or $\left\|x_i\right\| =\gamma.$
If $\left\|x_i\right\| < \gamma $, then there exists $z_{j_i} \in \Lambda_*$ such that
\begin{eqnarray}\label{eq905}
\left\|x_i\right\| \in \left[ z_{j_i},z_{j_i +1}\right).
\end{eqnarray}

Define new  function $\hat{x}_*(\cdot):\Omega\rightarrow \mathbb{R}^n$,  setting
\begin{eqnarray}\label{eq906}
\hat{x}_*(\xi)=\left\{
\begin{array}{lllll} \displaystyle
 \frac{x_i}{\left\|x_i\right\|}z_{j_i} \  , &  \ \mbox{if} \ &  0<\left\|x_i\right\| <\gamma ,\\
 x_i \  , & \ \mbox{if} \ &  \left\|x_i\right\|=0 \ \mbox{or} \  \left\|x_i\right\|=\gamma
\end{array}
\right.
\end{eqnarray}
where $\xi \in \Omega_i$, $i=1,2,\ldots, N$ and $z_{j_i}\in \Lambda_*$ is defined by (\ref{eq905}). It is not difficult to observe that $\left\|\hat{x}_*(\xi)\right\| \leq \left\|\hat{x}_0(\xi)\right\|$ for every $\xi \in \Omega$, and moreover, from (\ref{eq903}), (\ref{eq904}), (\ref{eq905}) and (\ref{eq906}) it follows that $\hat{x}_*(\cdot) \in B_{p}^{\gamma, \Delta, \delta}(r)$ and
\begin{eqnarray}\label{eq907}
\left\|\hat{x}_0(\xi)-\hat{x}_*(\xi)\right\| \leq  \delta
\end{eqnarray} for every $\xi \in \Omega.$ Now, let $\hat{y}_*(\cdot)\in \mathcal{F}_{p}^{\lambda,\gamma, \Delta, \delta}(r)$ be the image of $\hat{x}_*(\cdot) \in B_{p}^{\gamma, \Delta, \delta}(r)$ under operator (\ref{fo*}). Thus, (\ref{eq2a}) and (\ref{eq907})  imply that
\begin{eqnarray*}  \displaystyle
\left\| \hat{y}_*(\xi) -\hat{y}_0(\xi)\right\| \leq M(\lambda)\mu (\Omega)\delta
\end{eqnarray*} for every $\xi \in \Omega$. From the last inequality and  (\ref{phi}) we conclude that
\begin{eqnarray}\label{eq909}  \displaystyle
\left\| \hat{y}_*(\cdot) -\hat{y}_0(\cdot)\right\|_{q} \leq M(\lambda)\left[\mu (\Omega)\right]^{1+\frac{1}{q}}\delta =\varphi_{\lambda}(\delta).
\end{eqnarray}

Since $\hat{y}_0(\cdot)\in \mathcal{F}_{p}^{\lambda,\gamma, \Delta}(r)$ is arbitrarily chosen and $\hat{y}_*(\cdot)\in \mathcal{F}_{p}^{\lambda,\gamma, \Delta, \delta}(r)$, the inequality (\ref{eq909}) yields that
\begin{eqnarray}\label{eq910}  \displaystyle
\mathcal{F}_{p}^{\lambda,\gamma, \Delta}(r) \subset \mathcal{F}_{p}^{\lambda,\gamma, \Delta,\delta}(r)+ \varphi_{\lambda}(\delta)B_q(1).
\end{eqnarray}

The inclusion $\mathcal{F}_{p}^{\lambda,\gamma, \Delta, \delta}(r) \subset \mathcal{F}_{p}^{\lambda,\gamma, \Delta}(r)$ and (\ref{eq910}) imply
\begin{eqnarray}\label{eq911}  \displaystyle
h_{q}\left(\mathcal{F}_{p}^{\lambda,\gamma, \Delta}(r), \mathcal{F}_{p}^{\lambda,\gamma, \Delta,\delta}(r)\right) \leq \varphi_{\lambda}(\delta).
\end{eqnarray}

\vspace{3mm}

\emph{Step 5.} For given $\Delta >0,$ $\delta>0$,  finite $\Delta$-partition $\Lambda= \left\{\Omega_1,\Omega_2,\ldots, \Omega_N \right\}$ of the compact set $\Omega\subset \mathbb{R}^k$, uniform $\delta$-partition $\Lambda_* =\left\{0=z_0, z_1, \ldots , z_a=\gamma \right\}$ of the closed interval $[0,\gamma]$  and finite $\sigma$-net $E_{\sigma}=\left\{e_1,e_2, \ldots, e_c\right\}$ of $E=\left\{x\in \mathbb{R}^n: \left\|x\right\| =1\right\}$, the set $B_{p}^{\gamma, \Delta, \delta, \sigma}(r)$ is defined by (\ref{fin}). Denote
\begin{eqnarray*}
\mathcal{F}_{p}^{\lambda, \gamma, \Delta, \delta, \sigma}(r)= \left\{ F_{\lambda}(x(\cdot))|(\cdot): x(\cdot)\in B_{p}^{\gamma, \Delta, \delta, \sigma}(r)\right\}.
\end{eqnarray*}
Choose an arbitrary $\overline{y}_0(\cdot)\in \mathcal{F}_{p}^{\lambda,\gamma, \Delta,\delta}(r)$ which is the image of $\overline{x}_0(\cdot)\in B_{p}^{\gamma, \Delta,\delta}(r)$ under operator (\ref{fo*}). From inclusion $\overline{x}_0(\cdot)\in B_{p}^{\gamma, \Delta,\delta}(r)$ it follows that there exists $z_{j_i} \in \Lambda_*,$ $s_i \in E$ $(i=1,2,\ldots, N)$ such that
\begin{eqnarray}\label{eq1003}
\displaystyle \overline{x}_0(\xi) =z_{j_i}s_i, \ \xi \in \Omega_i,  \ i=1,2,\ldots , N
\end{eqnarray} where
\begin{eqnarray}\label{eq1004}
\displaystyle  \sum_{i=1}^{N} \mu(\Omega_i) z_{j_i}^p \leq r^p .
\end{eqnarray}

Since $s_i \in E$ for every $i=1,2,\ldots, N$, $E_{\sigma}$ is a finite $\sigma$-net on $E$, then for each $s_i \in E$ there exists $e_{l_i} \in E_{\sigma}$ such that
\begin{eqnarray}\label{eq1005}
\left\|s_i -e_{l_i} \right\| \leq \sigma.
\end{eqnarray}
Define function $\overline{x}_0(\cdot):\Omega \rightarrow \mathbb{R}^n$
setting
\begin{eqnarray}\label{eq1006*}
\displaystyle \overline{x}_*(\xi) =z_{j_i}e_{l_i}, \ \xi \in \Omega_i,  \ i=1,2,\ldots , N.
\end{eqnarray}
From  (\ref{eq1003}), (\ref{eq1004}), (\ref{eq1005}) and (\ref{eq1006*}) it follows that  $\overline{x}_*(\cdot)\in B_{p}^{\gamma, \Delta, \delta, \sigma}(r)$ and
\begin{eqnarray}\label{eq1006}
\displaystyle \left\| \overline{x}_*(\xi)- \overline{x}_0(\xi)\right\|  \leq z_{j_i} \left\|s_i-e_{l_i}\right\| \leq \gamma \sigma
\end{eqnarray} for every $\xi \in \Omega.$ Not let $\overline{y}_*(\cdot)\in \mathcal{F}_{p}^{\lambda,\gamma, \Delta,\delta,\sigma}(r)$ be the image of  $\overline{x}_*(\cdot)\in B_{p}^{\gamma, \Delta, \delta, \sigma}(r)$ under operator (\ref{fo*}). Then, (\ref{eq1006}) yields
\begin{eqnarray*}
\left\| \overline{y}_*(\xi)-\overline{y}_0(\xi)\right\| \leq M(\lambda)\mu(\Omega)  \gamma \sigma
\end{eqnarray*} for every $\xi \in \Omega$ and hence
\begin{eqnarray}\label{eq1008}
\left\| \overline{y}_*(\cdot)-\overline{y}_0(\cdot)\right\|_{q} \leq M(\lambda)\left[\mu(\Omega)\right]^{1+\frac{1}{q}}  \gamma \sigma =\alpha_{\lambda}(\gamma, \sigma),
\end{eqnarray} where $M(\lambda)$ is defined by (\ref{eq2a}), $\alpha_{\lambda}(\gamma, \sigma)$ is defined by (\ref{alfa}). Thus, for arbitrary chosen $\overline{y}_0(\cdot)\in \mathcal{F}_{p}^{\lambda,\gamma, \Delta,\delta}(r)$ there exists $\overline{y}_*(\cdot)\in \mathcal{F}_{p}^{\lambda,\gamma, \Delta,\delta,\sigma}(r)$ such that the inequality (\ref{eq1008}) is satisfied. This means that
\begin{eqnarray*}
\mathcal{F}_{p}^{\lambda,\gamma, \Delta,\delta}(r) \subset \mathcal{F}_{p}^{\lambda,\gamma, \Delta,\delta,\sigma}(r) +\alpha_{\lambda}(\gamma, \sigma) B_q(1).
\end{eqnarray*}

The last inclusion and inclusion $\mathcal{F}_{p}^{\lambda,\gamma, \Delta,\delta,\sigma}(r)\subset \mathcal{F}_{p}^{\lambda,\gamma, \Delta,\delta}(r)$
imply that
\begin{eqnarray}\label{eq1009}
h_{q}\left(\mathcal{F}_{p}^{\lambda,\gamma, \Delta,\delta,\sigma}(r), \mathcal{F}_{p}^{\lambda,\gamma, \Delta,\delta}(r)\right) \leq \alpha_{\lambda}(\gamma, \sigma).
\end{eqnarray}

\vspace{3mm}

\emph{Step 6.} Analogously to (\ref{eq1t}) it is not difficult to show that

\begin{eqnarray}\label{eq1101}
h_{q}\left(\mathcal{F}_{p}^{\lambda,\gamma, \Delta,\delta,\sigma}(r), \mathcal{F}_{p}^{\gamma, \Delta,\delta,\sigma}(r)\right) \leq \frac{\lambda}{2}
\end{eqnarray} where the set $\mathcal{F}_{p}^{\gamma, \Delta,\delta,\sigma}(r)$ is defined by (\ref{eq801}).

\vspace{3mm}

\emph{Step 7.} Now, the proof of the theorem follows from inequalities (\ref{eq1t}), (\ref{eq810}), (\ref{eq824}), (\ref{eq911}), (\ref{eq1009}) and (\ref{eq1101}).
\end{proof}

From Theorem \ref{teo1} it follows the validity of the following Corollary.
\begin{corollary} \label{cor1} For every $\varepsilon >0$ there exists $\lambda(\varepsilon) >0,$ $\gamma_*(\varepsilon)=\gamma (\varepsilon,\lambda(\varepsilon))>0,$ $\Delta_*(\varepsilon)=\Delta (\varepsilon,\lambda(\varepsilon))>0,$ $\delta_*(\varepsilon)=\delta (\varepsilon,\lambda(\varepsilon))>0$ and $\sigma_*(\varepsilon)=\sigma (\varepsilon,\lambda(\varepsilon),\gamma_*(\varepsilon))>0$ such that for every $\Delta $-partition of the compact set $\Omega\subset \mathbb{R}^k,$ $\delta$-partition of the closed interval $[0,\gamma]$ and $\sigma >0$ the inequality
\begin{eqnarray*}
h_{q}\left(\mathcal{F}_{p}(r),\mathcal{F}_{p}^{\gamma_*(\varepsilon), \Delta, \delta, \sigma}(r)\right) \leq \varepsilon
\end{eqnarray*}
is satisfied for each $\Delta \in (0,\Delta_* (\varepsilon)),$ $\delta \in (0, \delta_* (\varepsilon))$ and
$\sigma \in (0,\sigma_* (\varepsilon)).$
\end{corollary}
\begin{proof} Let us choose $\displaystyle \lambda(\varepsilon) =\frac{\varepsilon}{5}$ and fix it, and let
\begin{eqnarray*}
\gamma_*(\varepsilon)=\gamma(\varepsilon,\lambda(\varepsilon))=\left(\frac{5c_* M(\lambda(\varepsilon))}{\varepsilon}\right)^{\frac{1}{p-1}},
\end{eqnarray*}
$\Delta_*(\varepsilon)=\Delta(\varepsilon, \lambda(\varepsilon))>0$ be such that
\begin{eqnarray*}
\psi_{\lambda(\varepsilon)}(\Delta) \leq \frac{\varepsilon}{5}
\end{eqnarray*} for every $\Delta \in (0,\Delta_*(\varepsilon)),$
\begin{eqnarray*}
\delta_*(\varepsilon)=\delta(\varepsilon,\lambda(\varepsilon)) = \frac{\varepsilon}{5M(\lambda(\varepsilon))\left[\mu(\Omega)\right]^{1+\frac{1}{q}}},
\end{eqnarray*}
\begin{eqnarray*}
\sigma_*(\varepsilon)=\sigma(\varepsilon,\lambda(\varepsilon),\gamma_*(\varepsilon))= \frac{\varepsilon}{5M(\lambda(\varepsilon))\left[\mu (\Omega)\right]^{1+\frac{1}{q}} \gamma_*(\varepsilon)}.
\end{eqnarray*}

Now the proof of the corollary follows from  Theorem \ref{teo1}.
\end{proof}

\end{document}